\documentclass{amsart}
\usepackage{amsfonts,amssymb,amsmath,amsthm}
\usepackage{url}
\usepackage{enumerate}
\usepackage[all]{xy}
\usepackage{color}   
\newtheorem{thm}{Theorem}[section]
\newtheorem{lemma}[thm]{Lemma}
\newtheorem{prop}[thm]{Proposition}

\theoremstyle{definition}
\newtheorem{definition}[thm]{Definition}
\newtheorem{remark}[thm]{Remark}  
\newtheorem{example}[thm]{Example}  
  
\numberwithin{equation}{section}

\author{Seyed Ruhallah Ahmadi and Bruce Gilligan}
\address{
Department of Mathematics and Statistics, University of Regina,
Regina, Canada S4S 0A2}
\email{seyed.ruhala.ahmadi@gmail.com and gilligan@math.uregina.com}
\keywords{homogeneous K\"{a}hler manifold, non-compact dimension two}
\subjclass{Primary 32M10; Secondary 32Q15}   

\date{\today}   

\begin{document}

\title[Classification of K\"{a}hler $G/H$ with $d_{G/H}=2$]{Classification of  K\"{a}hler homogeneous  \\  
manifolds of non--compact dimension two}

\begin{abstract}    
Suppose $G$ is a connected complex Lie group and $H$ is a closed complex 
subgroup such that $X := G/H$ is K\"{a}hler and   
the codimension of the top non--vanishing homology group of $X$ with 
coefficients in $\mathbb Z_2$ is less than or equal to two.  
We show that $X$ is biholomorphic to a complex homogeneous manifold constructed using well--known   
``basic building blocks'', i.e., $\mathbb C$, $\mathbb C^*$, Cousin groups, and flag manifolds.    
For $H$ discrete the classification was presented in the first author's dissertation \cite{Ahm13}.        
\end{abstract}  

\maketitle

\section{Introduction}   \label{sect1}  

In this paper we consider complex homogeneous manifolds of the form $G/H$,   
where $G$ is a connected complex Lie group and $H$ is a closed complex subgroup of $G$.    
The existence of complex analytic objects on such a $G/H$, like non--constant holomorphic functions, 
plurisubharmonic functions and analytic hypersurfaces, is related to when $G/H$ could be K\"{a}hler.  
So the first question one might consider concerns the existence of K\"{a}hler structures and we restrict 
ourselves to that question here.       
The structure of compact K\"{a}hler homogeneous manifolds is now classical \cite{Mat57} and \cite{BR62}  
and the structure in the case of $G$--invariant metrics is also known \cite{DN88}.  
As a consequence, our investigations here concern non--compact complex homogeneous manifolds 
having a K\"{a}hler metric that is not necessarily $G$--invariant.     

\medskip   
Some results are known under restrictions on the type of group $G$ that is acting.      
The base of the holomorphic reduction of any complex solvmanifold is always Stein \cite{HO86}, where the proof uses 
some fundamental ideas in  \cite{Loe85}.  
For  $G$ a solvable complex Lie group and $G/H$ K\"{a}hler the fiber of the holomorphic reduction   
of $G/H$ is a Cousin group, see \cite{OR88} 
and the holomorphic reduction of a finite covering of $G/H$ is a principal Cousin group bundle, see \cite{GO08}.    
If $G$ is semisimple, then $G/H$ admits a K\"{a}hler structure if and only if $H$ is algebraic \cite{BO88}.     
For $G$ reductive there is the characterization that $G/H$ is K\"{a}hler if and only if $S\cdot H$ is closed in $G$    
and $S\cap H$ is an algebraic subgroup of $S$, a maximal semisimple subgroup of $G$, see \cite[Theorem 5.1]{GMO11}.          
There is also a result if $G$ is the direct product of its radical and a maximal semisimple subgroup    
under some additional assumptions on the isotropy subgroup and so on the structure of $X$ \cite{OR87}.         

\bigskip  
One way to proceed is to impose some topological restraints on $X$.   
In \cite{GOR89} we classified K\"{a}hler homogeneous manifolds  $X$ having more than one end       
by showing $X$ is either a product of a Cousin group and a flag manifold 
or $X$ admits a homogeneous fibration as a  $\mathbb C^*$--bundle over the 
product of a compact complex torus and a flag manifold.       
Now in the setting of proper actions of Lie groups Abels introduced the notion of non--compact dimension, 
see \cite{Abe76} and  \cite[\S 2]{Abe82}.        
We do not wish to assume that our Lie group actions are necessarily proper ones, so we take a  
dual approach and define the non--compact dimension $d_X$ of a connected smooth manifold to 
be the codimension of its top non--vanishing homology group with coefficients in $\mathbb Z_2$, see \S 2.   
Our goal  in this paper is to classify K\"{a}hler homogeneous manifolds $G/H$ with $d_{G/H} = 2$.   
All such spaces are holomorphic fiber bundles where the fibers and the bases of the bundles involved consist of    
Cousin groups, flag manifolds, $\mathbb C$, and $\mathbb C^*$.      
We now present the statement of our main result, where $T$ denotes a compact complex torus, $C$ a Cousin group, and $Q$ a flag manifold.

\begin{thm}  [Main Theorem]  
Suppose $X := G/H$ is K\"{a}hler, where $G$ is a connected complex Lie group 
and $H$ is  a closed complex subgroup of $G$.       
Then  $d_X = 2$ if and only if $X$ is one of the following:   

\medskip\noindent {\bf Case I:  $H$ discrete:}   
A finite covering of $X$ is biholomorphic to a product $ C \times A$, 
with $C$ a Cousin group, $A$ a Stein Abelian Lie group and $d_C + d_A = 2$.      

\medskip \noindent  
\noindent {\bf Case II:  $H$ is not discrete:}    
\begin{enumerate}   
\item Suppose ${\mathcal O}(X) = \mathbb C$ and let  $G/H \to G/N$ be its normalizer fibration.
\subitem (a)  $X$ is a $(\mathbb C^*)^k$--bundle over $C \times Q$ with $d_C + k = d_X = 2$.  
\subitem (b)  $X$ is $T \times G/N$ with ${\mathcal O}(G/N)=\mathbb C$ and   
$G/N$ fibers as a $\mathbb C$--bundle over a flag manifold;  
there are two subcases depending on whether $S$ acts transitively on $G/N$ or not.  
\item    Suppose ${\mathcal O}(X) \not= \mathbb C$ and let $G/H \to G/J$ be its holomorphic reduction.   
\subitem (a)  $G/J$ is Stein with $d_{G/J} = 2$ and $X = T \times Q \times G/J$, where   
\subsubitem (i) $G/J = \mathbb C$.  
\subsubitem (ii) $G/J$ is the 2--dimensional affine quadric
\subsubitem (iii) $G/J$ is the complement of a quadric curve in $\mathbb P_2$    
\subsubitem (iv)   $G/J =  (\mathbb C^*)^2$  
\subitem (b) $G/J$ is not Stein with $d_{G/J} = 2$; then   
$G/J$ is a $\mathbb C^*$--bundle over an affine cone minus its vertex and  $X = T \times G/J$.      
\subitem (c)  $d_{G/J} = 1$ with $G/J$ an affine cone minus its vertex and $d_{J/H} = 1$  
\subsubitem (i)  ${\mathcal O}(J/H) = \mathbb C$:  $X$ is a $\mathbb C^*$--bundle over $T\times Y$, where 
$Y$ is a flag manifold bundle over the affine cone minus its vertex.       
\subsubitem (ii)   ${\mathcal O}(J/H) \not= \mathbb C$:  this case does not occur.  
\end{enumerate}  
\end{thm}

The paper is organized as follows.  
In section two we gather a number of technical tools.   
In particular, we note that Proposition \ref{Cousinoverflag} deals with the setting where the fiber of the normalizer fibration is   
a Cousin group and its base is flag manifold.  
It is essential for Case II (1) (a) in the Main Theorem and can be used to simplify the proof when $d_X = 1$, see Remark \ref{BRGOR}.       
Section three is devoted to the case when the isotropy subgroup is discrete.      
Sections four and five deal with general isotropy and contain the proof of the main result when there are no non--constant 
holomorphic functions and when there are non--constant holomorphic functions, respectively.  
In the last section we present some examples.  % that illuminate some of the spaces that can occur.   

\medskip   
\thanks{This work was partially supported by an NSERC Discovery Grant.}
\thanks{We thank Prof. A. T. Huckleberry for his comments that led to improvements in the work.  
We also thank the Arbeitsgruppe Transformationsgruppen at the Ruhr--Universit\"{a}t Bochum for 
their kind hospitality during a sojourn at which time part of this paper was written.}

\section{Technical Tools}  

The purpose of this section is to gather together a number of definitions and basic tools that are subsequently needed.   
\subsection{Basic Notions}  

\begin{definition}  
A {\it Cousin group} is a complex Lie group $G$ with ${\mathcal O}(G) = \mathbb C$.  
The terminology {\it toroidal group} is also found in the literature.  
Every Cousin group is Abelian and is the quotient of $\mathbb C^n$ by a discrete subgroup having rank $n+k$ 
for some $k$ with $1 \le k \le n$.  
For details, we refer the reader to \cite{AK01}.       
\end{definition}  

\begin{definition}  
A {\it flag manifold} (the terminology {\it homogeneous rational manifold} is also in common usage)   
is a homogeneous manifold of the form $S/P$, where $S$ is a connected semisimple complex Lie group and $P$ is a parabolic 
subgroup of $S$.  
One source concerning the structure of flag manifolds is \cite[\S 3.1]{Akh95}.   
\end{definition}   

\begin{definition}  
For $X$ a connected (real) smooth manifold we define 
\[  
    d_{X} \; := \; \dim_{\mathbb R} X \; - \; 
    \min \{ \ r \ | \ H_{k}(X,\mathbb Z_2) \ = \ 0 \quad \forall \quad k \ > \ r  \ \} ,  
\]    
i.e., $d_{X}$ is the codimension of the top non--vanishing homology group of $X$ 
with coefficients in $\mathbb Z_2$.    
We call $d_X$ the {\it non--compact dimension} of $X$.    
\end{definition}   

\begin{prop}  \label{SerreStein}   
Suppose $X$ is a connected Stein manifold.  
Then  $\dim_{\mathbb C} X \le d_{X}$.
\end{prop}   

\begin{proof}  For $X$ Stein one has $H_k(X) = 0$ for all $k > \dim_{\mathbb C} X$ by \cite{Ser53} .  \end{proof}  

\subsection{Fibration Methods}   
Throughout we make use of a number of fibrations that are now classical.     
\begin{enumerate} 
\item {\bf Normalizer Fibration:}  Given $G/H$ let $N = N_G(H^0)$ be the normalizer in $G$ of 
the connected component of the identity $H^0$ of $H$.  
Since $H$ normalizes $H^0$, we have $H \subset N$ and the {\it normalizer fibration} $G/H \to G/N$.   
\item {\bf Holomorphic Reduction:}  Given $G/H$ we set 
$ J := \{ \ g \in G \ | \  f(gH) \; = \; f(eH) \mbox{ for all }  f \in {\mathcal O}(G/H) \ \} $.  
Then $J$ is a closed complex subgroup of $G$ containing $H$ and we 
call the fibration $p:G/H \to G/J$ the {\it holomorphic reduction} of $G/H$.   
By construction $G/J$ is holomorphically separable and  
 ${\mathcal O}(G/H) \cong p^*({\mathcal O}( G/J))$.   
\end{enumerate}

Suppose a manifold $X$ admits a locally trivial fiber bundle $ X \stackrel{F}{\to} B$ with $F$ and $ B$ connected smooth manifolds.   
One would then like to know how $d_F$ and $d_B$ are related to $d_X$ whenever possible.    
The following result was proved in \S 2 in \cite{AG94} using spectral sequences.          

\begin{lemma} [The Fibration Lemma]  
Suppose $ X \stackrel{F}{\to} B$ is a locally trivial fiber bundle with $X,F,B$ smooth manifolds.  
Then    
\begin{enumerate}  
\item if the bundle is orientable (e.g., if $\pi_1(B) = 0$), then $d_X = d_F + d_B$.    
\item if $B$ has the homotopy type of a $q$--dimensional CW complex, then 
$d_X \ge d_F + (\dim B - q)$.    
\item if $B$ is homotopy equivalent to a compact manifold, then $d_X \ge d_F + d_B$.  
\end{enumerate}  
\end{lemma}  

\begin{remark}    
If $B$ is homogeneous, then one knows that $B$ is homotopy equivalent to a compact manifold if:        
\begin{enumerate} 
\item the isotropy subgroup of $B$ has finitely many connected components \cite{Mos55}; 
e.g., in an algebraic setting,       
\item if $B$ is a solvmanifold \cite{Mos54}; indeed, every solvmanifold is a vector 
bundle over a compact solvmanifold \cite{AT70}.  
\end{enumerate}      
\end{remark}

\subsection{Special case of a question of Akhiezer}   

Later we will need a result that is based on \cite[Lemma 8]{AG94}.  
Since that Lemma was stated in a way suitable for its particular application in \cite{AG94}, 
we reformulate it in a form suitable for the present context.       

\begin{lemma} [Lemma 8 \cite{AG94}] \label{lem8}   
Let $G$ be a connected, simply connected complex Lie group with Levi-Malcev decomposition $G=S\ltimes R$ with 
$\dim_{\mathbb C} R=2$ and $\Gamma$ a discrete subgroup of $G$ such that $X=G/\Gamma$ is K\"{a}hler.    
Then $\Gamma$ is contained in a subgroup of $G$    
of the form $A \ltimes R$, where $A$ is a proper algebraic subgroup of $S$.   
\end{lemma}

This has the following consequence which we use later.  

\begin{thm}   \label{AG942} 
Suppose $G$ is a connected, simply connected, complex Lie group with Levi-Malcev decomposition $G=S\ltimes R$ with 
$\dim_{\mathbb C} R=2$. 
Let $\Gamma$ be a discrete subgroup of $G$ such that $X=G/\Gamma$ is K\"{a}hler,   
$\Gamma$ is not contained in a proper parabolic subgroup of $G$ and $\mathcal{O}(G/\Gamma) 
\simeq \mathbb C$.       
Then  $S=\{e\}$, i.e., $G$ is solvable.   
\end{thm}  

\begin{proof}  
The radical orbits are closed, e.g., see \cite{Gil04}.       
By lemma \ref{lem8} the subgroup $\Gamma$ is contained in a proper subgroup of $G$ of the form 
$A \ltimes R$, where $A$ is an algebraic subgroup of $S$.  
Thus there are fibrations    
\[  
         G/\Gamma \; \longrightarrow \;  G/R \cdot\Gamma  \; \longrightarrow \; S/A ,  
\]  
where the base $G/R\cdot\Gamma = S/\Lambda$ with $\Lambda := S \cap R.\Gamma$.  
If $A$ is reductive, then $S/A$ is Stein and we get non-constant holomorphic functions 
on $X$ as pullbacks using the above fibrations.    
But this contradicts the assumption that $\mathcal{O}(X) \simeq \mathbb C$.  
If $A$ is not reductive then \cite[Theorem, \S 30.1]{Hum75} applies and     
$A$ is contained in a proper parabolic subgroup of $S$.      
But this implies $\Gamma$ is too, thus contradicting the assumption that this is not the case.  
\end{proof}  
  
\subsection{The algebraic setting revisited}  
Throughout we repeatedly use two results of Akhiezer concerning the invariant $d_X$ in the setting where 
$X=G/H$ and $G$ is a connected linear algebraic group over $\mathbb C$ and $H$ 
is an algebraic subgroup of $G$.   
For the convenience of the reader we now state these here.             

\begin{thm} [$d=1$ \cite{Akh77};  $d=2$ \cite{Akh83}]  \label{alggps}  
Suppose $G$ is a connected linear algebraic group over $\mathbb C$ and $H$ is an   
algebraic subgroup of $G$ and $X := G/H$.  
\begin{enumerate}  
\item $d_X = 1 \Longrightarrow$ $H$ is contained in a parabolic subgroup $P$ of $G$ with $P/H = \mathbb C^*$.  
\item $d_X = 2 \Longrightarrow$ $H$ is contained in a parabolic subgroup $P$ of $G$ with $P/H$ being:    
\subitem (a) $ \mathbb C$
\subitem (b) the affine quadric $Q_2$ 
\subitem (c) the complement of a quadric curve in $\mathbb P_2$    
\subitem (c) $(\mathbb C^*)^2$     
\end{enumerate}  
\end{thm}     

%%%%%%%%%%%%%%%%%%%%%%%%%%%%%%%%%%%%   
\subsection{Cousin group bundles over flag manifolds}  

In this section we prove a result concerning the structure of K\"{a}hler homogeneous manifolds  
whose normalizer fibrations are Cousin group bundles over flag manifolds,      
where there is no assumption about the invariant $d$.   
In order to do this we will show that one can reduce to the case where a complex reductive group is acting transitively and   
employ some now classical details about the structure of parabolic subgroups, e.g., see \cite{Akh95} or \cite{FHW06}.   
A crucial point in the setting of interest is the fact that all $S$--orbits are closed  and 
have algebraic isotropy \cite[Theorem 5.1]{GMO11}.      

\begin{prop}  \label{Cousinoverflag}  
Suppose $G/H$ is a K\"{a}hler homogeneous manifold whose normalizer fibration $G/H \to G/N$ has fiber $C := N/H$ a Cousin 
group and base $Q := G/N$ a flag manifold.   
Then there exists a closed complex subgroup $J$ of $N$ containing $H$ such that 
the fibration $G/H \to G/J$ realizes $X$ as a  $(\mathbb C^*)^k$--bundle over a product 
$G/J = Q \times Y$, where $Y$ is a Cousin group  with $d_Y = d_X -  k $.   
\end{prop}     

\begin{proof}   
Our first task is to show that there is a reductive complex Lie group acting holomorphically and transitively on $X$.   
Write $C = \mathbb C^k/\Gamma$ and note that there exists a subgroup $\widehat{\Gamma} < \Gamma$ 
such that $\widehat{C} := \mathbb C^q/\widehat{\Gamma}$ is isomorphic to $(\mathbb C^*)^q$  
and is a covering group of $C$, e.g., see  \cite[\S 1.1]{AK01}.         
In particular, the reductive complex Lie group $\widehat{G} :=  S\times \widehat{C}$ acts transitively on $X$.         
We drop the hats from now on and assume, by going to a finite covering, if necessary, that 
$G= S \times Z$ is a reductive complex Lie group, where $Z \cong (\mathbb C^*)^q$ is the center of $G$   
and $S$ is a maximal semisimple subgroup.     

\bigskip   
Let $x_0\in X$ be the neutral point, $z_0\in G/N$ be its projection in the base, and $F_{z_0}$ be the fiber over the point $z_0$.     
Now the $S$--orbit $S\cdot H/H = S/S\cap H$ is closed in $X$  and $S\cap H$ is an algebraic subgroup of $S$ \cite[Theorem 5.1]{GMO11}.   
So there is an induced fibration  
\begin{eqnarray} \label{Sorbit}  
    \begin{array}{ccl}  S/S\cap H & \hookrightarrow & G/H \\  
                                   & & \\  
                                    A \downarrow & & \downarrow F_{z_0} \cong C \\  
                                    & & \\  
                                    S/P & = & G/N  \end{array}     
\end{eqnarray} 
Since the center normalizes any subgroup, we have $N = P \times Z$ with $P$ a parabolic subgroup of $S$.   
Now any parabolic group $P$ can be written as a semidirect product of its unipotent radical $U_P$ and a subgroup $L_P$.  
Further, $L_P$ is the centralizer in $S$ of a subgroup $C_P$ of $L_P$.  
In particular, it follows that $C_P$ is the center of $L_P$, e.g., see \cite{Akh95}.  
In passing, we also note that the commutator subgroup of $L_P$ is semisimple and thus $L_P$ is a complex reductive Lie group.    

\bigskip  
The bundle in (\ref{Sorbit}) is defined by a representation 
$\rho:  N  \longrightarrow {\rm Aut}^0(C)  \cong  C$     
and the group  $A =  \rho(P)$ lies in the connected component of the identity of the automorphism group 
of $C$ because $P$ is connected.          
Since $C$ is Abelian, $\rho|_P$ factors through the canonical projection from $P$ to $P/P'$.  
Because every parabolic subgroup contains a maximal torus, every root space in 
the parabolic is in its commutator subgroup.  
Also $C_P$ is not in the commutator subgroup, while the commutator subgroup of $L_P$ is.
As a consequence, $P/P'  \cong (\mathbb C^*)^p$, with $p = \dim C_P$, e.g., see also \cite[Proposition 8, \S 3.1]{Akh95}.       
Now the $S$--orbits in $G/H$ are closed and the $C_P$--orbit in the typical fiber $N/H$ 
through the neutral point is the intersection of $N/H$ with the corresponding $S$--orbit.     
So this $C_P$--orbit is closed and complex and is an algebraic quotient of the group $C_P$ biholomorphic to $(\mathbb C^*)^k$ for some $k \le p$.   
In this reductive setting $S$ is a normal subgroup of $G$ and so all $S$--orbits are biholomorphic and fiber over 
$G/N$ with fiber biholomorphic to $(\mathbb C^*)^k$.   
If $k=0$, then $X = C \times G/N$, by the argument given in the last paragraph of the proof.  

\bigskip   
Assume $k > 0$.   
Let $N \; \stackrel{\pi_1}{\longrightarrow} \; N/H^0 \; \stackrel{\pi_2}{\longrightarrow} \; N/H \; \cong \; C  $,  
where $\pi_1$ is the canonical homomorphism with $H^0$ normal in $N$ and $\pi_2$ is a covering homomorphism, 
since $H/H^0$ is a (normal) discrete subgroup of the Abelian group $N/H^0$.  
Set 
\[  
     J \; := \;  \pi^{-1}_1 \circ \pi^{-1}_2 (\rho(C_P)) .   
\]  
It follows that $J$ is a closed complex Lie 
subgroup of $G$ contained in $N$, since both $\pi_1$ and $\pi_2$ are holomorphic Lie group homomorphisms,   
that $J$ contains $H$ by its definition,  and that the $J$--orbit through the neutral point is the $C_P$--orbit.    
This yields an intermediate fibration 
\[  
         G/H \;  \stackrel{ (\mathbb C^*)^k}\longrightarrow \;  G/J \; \longrightarrow \; G/N
\]   
with the first fiber $(\mathbb C^*)^k \cong A$.    

\bigskip  
We claim that the bundle $G/J \to G/N$ is holomorphically trivial.  
In order to see this note that the $P$--action on the neutral fiber of the bundle $G/J \to G/N$ is trivial.  
Otherwise, the dimension of the $P$--orbit in $F_{z_0}$ would be bigger than $k$, as assumed above, and this 
would be a contradiction.   
As a consequence, all $S$--orbits in $G/J$ are sections of the bundle $G/J \to G/N$.   
Since $G/N$ is simply connected, this bundle is holomorphically trivial, i.e., $G/J = (C/A) \times S/P$.  
Because $C$ is a Cousin group, $Y := C/A$ is also a Cousin group.     
The statement about the topological invariant follows because  
$d_{C/A} = d_{G/J}$, since $S/P$ is compact, and $d_{G/J} = d_X - k$.         
\end{proof}  

\begin{remark}  \label{BRGOR}   
The case $d_X = 1$ is treated in \cite[Proposition 5]{GOR89}, where it is assumed that $X$ has more than one end.  
For $X$ K\"{a}hler this is equivalent to $d_X =1$.   
\end{remark}

\section{The Discrete Case}  

Throughout this section we assume that $X = G/\Gamma$ is K\"{a}hler with $d_X = 2$,   
where $G$ is a connected, simply connected, complex Lie group and $\Gamma$ is a  discrete subgroup of $G$.  
We first show that $G$ is solvable.  
Then we prove that a finite covering of such an $X$ is biholomorphic to a product $C \times A$, where $C$ is a 
Cousin group and $A$ is a holomorphically separable complex Abelian Lie group.  

\subsection{The reduction to solvable groups} 

We first handle the case when the K\"{a}hler homogeneous manifold has no non--constant 
holomorphic functions.       

\begin{lemma}  \label{disone}  
Assume $\Gamma$ is a discrete subgroup of a connected simply connected complex Lie group $G$ 
that is not contained in a proper parabolic subgroup of $G$, 
with $X := G/\Gamma$ K\"{a}hler, ${\mathcal O}(X) = \mathbb C$, and $d_X \le 2$.  
Then $G$ is solvable. 
\end{lemma}  

\begin{proof}  
Assume $G = S \ltimes R$ is a Levi decomposition.  
Since the $R$--orbits are closed, we have a fibration 
\[  
         G/\Gamma \; \longrightarrow \; G/R\cdot\Gamma \; = \; S/\Lambda ,  
\]   
where $\Lambda := S \cap R\cdot\Gamma$ is Zariski dense and discrete in $S$, e.g., see \cite{Gil04}.  
Now if $\mathcal{O}(R\cdot\Gamma/\Gamma) = \mathbb C$, then the result was proved in \cite{Gil04}.  
Otherwise, let 
\[ 
          R\cdot\Gamma/\Gamma \; \longrightarrow R\cdot\Gamma/J \; =: \; Y
\]  
be the holomorphic reduction.  
Then $Y$ is holomorphically separable and since $R$ acts transitively on $Y$,  
it follows that $Y$ is Stein \cite{HO86}.  
But $2 = d_X \ge d_Y \ge \dim_{\mathbb C} Y$.  
Further $J^0$ is a normal subgroup of $G$.  
As a consequence, the quotient group $\widehat{R} := R/J^0$ has complex dimension one or two.  
If $\dim \widehat{R} = 1$, then $\widehat{G} := S \times \widehat{R}$ is a product and this implies $S = \{ e \}$ 
by \cite{OR87}.  
If $\dim \widehat{R} = 2$, then $\widehat{G}$ is either a product, see \cite{OR87} again, or is   
a non--trivial semidirect product.  
In the latter case the result now follows by Theorem \ref{AG942}.     
\end{proof}  

In the next Proposition we reduce to the setting where the Levi factor is $SL(2,\mathbb C)$.  
We first prove a technical Lemma in that setting.   

\begin{lemma}  \label{distwo}  
Suppose $G/\Gamma$ is K\"{a}hler and $d_{G/\Gamma} \le 2$, where $\Gamma$ is a discrete subgroup 
of a connected, complex Lie group of the form $G = SL(2,\mathbb C) \ltimes R$ with $R$ the radical of $G$.       
Then $\Gamma$ is not contained in a proper parabolic subgroup of $G$.  
\end{lemma}  

\begin{proof}  
Assume the contrary, i.e., that $\Gamma$ is contained in a proper parabolic subgroup 
and let $P$ be a maximal such subgroup of $G$.  
Note that $P$ is isomorphic to $B \ltimes R$, where $B$ is a Borel subgroup of $SL(2,\mathbb C)$.  
Let $P/\Gamma \to P/J$ be the holomorphic reduction.  
Then $P/\Gamma$ is a Cousin group \cite{OR88} and $P/J$ is Stein \cite{HO86}.    
Note that $J \not= P$, since otherwise $P$ would be Abelian, a contradiction.  
The Fibration Lemma and Proposition \ref{SerreStein} imply  $\dim_{\mathbb C} P/J = 1$ or 2.  
So $P/J$ is biholomorphic to $\mathbb C, \mathbb C^*, \mathbb C^*\times\mathbb C^*$, or the complex Klein bottle.  

\bigskip
In the first two cases $P/J$  is equivariantly embeddable in $\mathbb P_1$ and by \cite{Kod54} 
it follows that $G/J$ is K\"{a}hler.  
In the latter two cases the fiber $J/\Gamma$ is compact by the Fibration Lemma and we can push 
down the K\"{a}hler metric on $X$ to obtain a K\"{a}hler metric on $G/J$, e.g., see \cite{Bla56}.     
In particular, the $S$--orbit $S/(S\cap J)$ in $G/J$ is K\"{a}hler and so its isotropy $S\cap J$ is algebraic \cite{BO88}.    
Now consider the diagram    
\[  \begin{array}{ccccc} 
       G/\Gamma & \stackrel{F}{\longrightarrow} & G/J & \stackrel{Y}{\longrightarrow} & G/P \\  
       & & & &  ||  \\ 
       \cup & &\cup  & & \mathbb P_1 \\   
       & & & & || \\   
        S/S\cap\Gamma & \stackrel{F_S}{\longrightarrow} & S/S\cap J & \stackrel{Z}{\longrightarrow} & S/B \end{array}  
\]    
Note that because $Y := P/J$ is noncompact and $d_{G/\Gamma}=2$, it follows    
from the Fibration Lemma that either $d_F = 1$ or $F$ is compact.  
Since $F$ is an Abelian Lie group, it is clear that $d_{F_S} \le d_F$.   

\bigskip 
The following enumerate, up to isomorphism, the algebraic subgroups of $B$ and in each case 
we derive a contradiction.        
\begin{enumerate}  
\item $\dim_{\mathbb C} S \cap J = 2$.  
Then $S\cap J = B$.     
This yields the contradiction $d_{S/S\cap\Gamma} \le d_{F_S} + d_{S/B} =1+0=1 < 3 =   
d_{S/S\cap\Gamma}$, since $S\cap \Gamma$ is finite.  
\item  $\dim_{\mathbb C} S \cap J = 1$.  
\subitem (a) If $S\cap J = \mathbb C^*$, then $S/S\cap J$ is an affine quadric or the complement 
of a quadric curve in $\mathbb P_2$ and thus $Z = \mathbb C$.  
So $P/J \not= \mathbb C^*$ and is either $\mathbb C$ or $(\mathbb C^*)^2$, i.e., $d_{P/J}=2$.  
Then the Fibration Lemma implies $F$ is compact and, since $F_S$ is closed in $F$, it must also 
be compact.  
But this forces $S \cap\Gamma$ to be an infinite subgroup of $S\cap J$ which is a contradiction.  
\subitem (b) If $S\cap J = \mathbb C$, then $S/S\cap J$ is a finite quotient of $\mathbb C^2 \setminus \{ (0,0)\}$   
and so $Z = \mathbb C^*$.  
Now $P/J = \mathbb C,   (\mathbb C^*)^2$ or $\mathbb C^*$.  
In the first two instances $F$ would be compact and we get the same contradiction as in (a).  
In the last case $d_F =1$ by the Fibration Lemma and $F_S$ is either compact or $\mathbb C^*$.  
Again $S\cap\Gamma$ is infinite with the same contradiction as in (a).  
\item $\dim_{\mathbb C} S \cap J = 0$.  
Here $S\cap J$ is finite, since it is an algebraic subgroup of $B$.  
Then $\dim S/S\cap J = 3$ and we see that $\dim G/J = 3$, since we know $\dim G/P = 1$ and 
$\dim P/J \le 2$.  
Then $P/J = (\mathbb C^*)^2$ and, since $S/S\cap J$ is both open and closed in $G/J$, 
it follows that $S/S\cap J = G/J$ and $d_{S/S\cap J} = 2$.    
But $F$ is compact and thus so is $F_S$ and we get the contradiction that 
$d_{S/S\cap\Gamma} = 2 < 3= d_{S/S\cap\Gamma}$.  
\end{enumerate}   
As a consequence, $\Gamma$ is not contained in a proper parabolic subgroup of $G$.  
\end{proof}  

\begin{prop}  \label{solvgp}  
Suppose $G/\Gamma$ is K\"{a}hler with $d_{G/\Gamma} \le 2$.  
Then $G$ is solvable. 
\end{prop}  

\begin{proof} 
First note that $G$ cannot be semisimple.  
If that were so, then $\Gamma$ would be algebraic, hence finite and 
thus $G/\Gamma$ would be Stein.  
But then $2 = d_{G/\Gamma} \ge \dim_{\mathbb C} G/\Gamma = \dim_{\mathbb C} G$ 
which is a contradiction, since necessarily $\dim_{\mathbb C} G \ge 3$ for any complex 
semisimple Lie group $G$.         

\bigskip 
So assume $G = S\ltimes R$ is mixed.   
The proof is by induction on the dimension of $G$.  
Now if a proper parabolic subgroup of $G$ contains $\Gamma$, then a maximal one does too, 
is solvable by induction and so has the special form $B \ltimes R$, where $B$ is isomorphic to 
a Borel subgroup of $S = SL(2,\mathbb C)$. 
But this is impossible by Lemma \ref{distwo}.  

\bigskip  
Lemma \ref{disone} handles the case $\mathcal{O}(G/\Gamma)=\mathbb C$.    
So we assume $\mathcal{O}(G/\Gamma)\not=\mathbb C$ with holomorphic reduction $G/\Gamma \to G/J$. 
The Main Result in \cite{AG94} gives the following list of the possibilities for the base $G/J$:    
\begin{enumerate} 
\item $\mathbb C$    
\item affine quadric $Q_2$     
\item $\mathbb P_2 \setminus Q$, where $Q$ is quadric curve     
\item an affine cone minus its vertex    
\item $\mathbb C^*$--bundle over an affine cone minus its vertex
\end{enumerate}  
 
In case (1) the bundle is holomorphically trivial, with compact fiber a torus and the group that is 
acting effectively is solvable. 
In cases (2) and (3) we have fibrations $G/\Gamma \to G/J \to G/P = \mathbb P_1$ 
and so $\Gamma$ is contained in a proper parabolic subgroup of $G$, contradicting what was shown in the previous paragraph.     

\bigskip 
In order to handle cases (4) and (5) we recall that an affine cone minus its vertex fibers equivariantly 
as a $\mathbb C^*$--bundle over  a flag manifold.  
Thus we get fibrations 
\[  
         G/\Gamma \; \longrightarrow \; G/J \; \longrightarrow \; G/P .  
\]   
Note that it cannot be the case that $G \not= P$, since then $\Gamma$ would be contained in a proper parabolic 
subgroup, a possibility that has been ruled out.     
So $G=P$ and $G/J$ (or a 2--1 covering) is biholomorphic to $\mathbb C^*$ or $(\mathbb C^*)^2$.  
In the second case the fiber $J/\Gamma$ is compact and thus a torus and thus $G$ is solvable. 
If $G/J = \mathbb C^*$, then $J/\Gamma$ is K\"{a}hler with $\dim J < \dim G$ and $d_{J/\Gamma}=1$ by the Fibration Lemma.     
By induction $J$ is solvable and so $G$ is solvable too, because $G/J = \mathbb C^*$.   
\end{proof} 

\subsection{A product decomposition}  
In order to prove the classification we will have need of the next splitting result.  

\begin{prop} \label{torbdle}  
Suppose $G$ is a connected, simply connected solvable complex Lie group that contains a discrete subgroup 
$\Gamma$ such that $G/\Gamma$ is K\"{a}hler and has holomorphic reduction $G/\Gamma \to G/J$  
with base $(\mathbb C^*)^2$ and fiber a torus.  
Then a finite covering of $G/\Gamma$ is biholomorphic to a product.  
\end{prop}  

\begin{proof} 
First assume that $J^0$ is normal in $G$ and let $\alpha: G \to G/J^0$ be the quotient homomorphism  
with differential $d\alpha: {\mathfrak g} \to {\mathfrak g}/{\mathfrak j}$.    
Then $G/J^0$ is a two dimensional complex Lie group that is either Abelian or solvable.  
In the Abelian case $G_0 := \alpha^{-1}(S^1 \times S^1)$ is a subgroup of $G$ that has 
compact orbits in $X$, since these orbits fiber as torus bundles over $S^1 \times S^1$ in the base.   
The result now for $(\Gamma, G_0, G)$ follows from  \cite[Theorem 6.14]{GO11}, since this triple is a CRS.  

\bigskip  
Next assume that $G/J^0$ is isomorphic to the two dimensional Borel group $B$ with Lie algebra $\mathfrak b$.     
Then $\mathfrak n := d\alpha^{-1}(\mathfrak n_{\mathfrak b})$ is the nilradical of $G$.  
Let $N$ denote the corresponding connected Lie subgroup of $G$.  
Now choose $\gamma_N \in \Gamma_N := N\cap\Gamma$ such that $\alpha(\gamma_N) \not= e$.  
There exists $x\in\mathfrak n$ such that $\exp(x)=\gamma_N$.  
Let $U$ be the connected Lie group corresponding to $\langle \gamma_N \rangle_{\mathbb C}$.  
Since $\Gamma$ centralizes $J^0$ (see  \cite[Theorem 1]{GO08}), it follows that 
$\mathfrak n = \mathfrak u \oplus \mathfrak j$ and $N= U \times J^0$ is Abelian.  
Set $\Gamma_U := \Gamma \cap U$ and $\Gamma_J :=\Gamma\cap J^0$.  
Then $N/\Gamma_N = U/\Gamma_U \times J^0/\Gamma_J$.  

\bigskip  
Since $\Gamma/\Gamma_N = \mathbb Z$, we may choose $\gamma \in \Gamma$ such that 
$\gamma$ projects to a generator of $\Gamma/\Gamma_N$.  
Also set $A := \exp(\langle w \rangle_{\mathbb C})$ for fixed $w \in \mathfrak g \setminus \mathfrak n$.  
Since $A$ is complementary to $N$, we have $G = A \ltimes N$.  
Now there exist $\gamma_A\in A$ and $\gamma_N\in N$ such that $\gamma = \gamma_A \cdot \gamma_N$.  
Both $\gamma$ and $\gamma_N$ centralize $J^0$ and thus $\gamma_A$ does too.  
Also $\gamma_A = \exp(h)$ for some $h = sw$ with $s\in\mathbb C$.  
Therefore, 
\begin{eqnarray} \label{central}
      [h,\mathfrak j] \; = \; 0 .   
\end{eqnarray}   

\bigskip  
Since $\mathfrak a + \mathfrak u$ is isomorphic to $\mathfrak b = \mathfrak g/\mathfrak j$ as a vector space,   
there exists $e\in\mathfrak u$ such that 
\[  
             [d\alpha(h),d\alpha(e)] \; = \; 2 d\alpha(e) .  
\]  
Let $\{ e_1, \ldots, e_{n-2}\}$ be a basis for $\mathfrak j$.  
There exist structure constants $a_i$ so that 
\[  
           [h,e] \; = \; 2e \; + \; \sum_{i=1}^{n-2} a_i e_i 
\]  
and the remaining structure constants are all 0 by (\ref{central}).   
Note that, conversely, any choice of the structure constants $a_i$ determines a solvable Lie algebra $\mathfrak g$ 
and the corresponding connected simply--connected complex Lie group $G = A \ltimes N$.  

\bigskip 
We now compute the action of $\gamma_A\in A$ on $N$ by conjugation.   
In order to do this note that the adjoing representation restricted to $\mathfrak n$, i.e., the map ${\rm ad}_h : \mathfrak n \to \mathfrak n$, 
is expressed by the matrix 
\[  
       M \; := \; [{\rm ad}_h] \; = \; \left( \begin{array}{cccc}  2 & 0 & \ldots & 0 \\  
                              a_1 & 0 & \ldots & 0 \\  
                              \vdots & \vdots & \ddots & \vdots \\ 
                              a_{n-2} & 0 & \ldots & 0 \end{array} \right)   
\]  
So the action of $A$ on $N$ is through the one parameter group of linear transformations $ t \mapsto e^{tM}$ for $t\in\mathbb C$.  
For $k\geq 1$
\[   
              (tM)^k  \; = \; \frac{1}{2} (2t)^k M,
\]   
and it follows that     
\[   
      e^{tM} \; = \; \frac{1}{2}(e^{2t}-1)M \; + \; {\rm Id} .   
\]   
Now the projection of the element $\gamma_A$ 
acts trivially on the base $Y = G/J$, so $t=\pi i k$ where $k\in\mathbb Z$. 
Hence $\gamma_A$ acts trivially on U.
Also $\gamma_N$ acts trivially on $N$, since $N$ is Abelian.   
Thus $\gamma$ acts trivially on $N$ and as a consequence, 
although $G$ is a non-Abelian solvable group
the manifold $X=G/\Gamma$ is just the quotient of $\mathbb C^n$ by a discrete additive subgroup of rank $2n-2$. 
Its holomorphic reduction is the original torus bundle which, since we are now dealing with the Abelian case, is trivial.

\medskip  
Now assume $J^0$ is not normal in $G$, set $N := N_G(J^0)$, and let 
$
    G/J \; \xrightarrow{N/J} \; G/N
$
be the normalizer fibration.
Since the base $G/N$ of the normalizer fibration is an orbit in some projective space,     
$G/N$ is holomorphically separable and thus Stein \cite{HO86}.         
Since we also have $2 \ge d_{G/N} \ge \dim_{\mathbb C} G/N$ we see that 
$G/N \cong \mathbb C$, $\mathbb C^*$ or $\mathbb C^*\times \mathbb C^*$.      
We claim that we must have $G/N = \mathbb C^*$, i.e., we have to eliminate the other two possiblities.  
Assume $G/N \cong \mathbb C$.   
Since $d_X \leq 2$ the Fibration Lemma implies $d_{N/J} = 0$, i.e., $N/J$ is biholomorphic to a torus $T$. 
Thus $G/J = T\times \mathbb C$.
However, we assume that $G/J = \mathbb C^*\times\mathbb C^*$ and 
this gives us a contradiction.
Now assume $G/N \cong \mathbb C^* \times \mathbb C^*$.
By Chevalley's theorem \cite{Che51} the commutator group $G'$ acts algebraically.    
Hence the $G'$--orbits are closed and one gets the commutator fibration 
$\displaystyle   
        G/N \xrightarrow{\mathbb C} G/G'\cdot N.
$
Since $G$ is solvable, it follows that $G'$ is unipotent and the $G'$--orbits are cells, i.e.,  $G'\cdot N/N \cong \mathbb C$.
By the Fibration Lemma the base of the commutator fibration is a torus. 
But it is proved in \cite{HO81}  
that the base of a commutator fibration is always Stein which is a contradiction.
This proves the claim that $G/N \cong \mathbb C^*$ and by the Fibration Lemma $d_{N/J} = 1$ and hence $N/J = \mathbb C^*$ 

\medskip    
Since $G$ is simply connected, $G$ admits a Hochschild-Mostow hull (\cite{HM64}),   
i.e., there exists a solvable linear algebraic group 
\[
       G_a \; = \; (\mathbb C^*)^k \ltimes G
\]
that contains $G$ as a Zariski dense, topologically closed, normal complex subgroup.   
By passing to a subgroup of finite index we may, without loss of generality,     
assume the Zariski closure $G_a (\Gamma)$ of $\Gamma$ in $G_a$ is connected.   
Then $G_a (\Gamma) \supseteq J^0$ and $G_a (\Gamma)$ is nilpotent  \cite{GO08}.   
Let $\pi: \widehat{G_a} (\Gamma) \to G_a (\Gamma)$ be the universal covering and     
set $\widehat{\Gamma} := \pi^{-1}(\Gamma)$.       
Since $\widehat{G_a} (\Gamma)$ is a simply connected, nilpotent, complex Lie group, the exponential map from the 
Lie algebra ${\mathfrak g}_a (\Gamma)$ to $\widehat{G_a} (\Gamma)$ is bijective.   
For any subset of $\widehat{G_a} (\Gamma)$ and, in particular for the subgroup       
$\widehat{\Gamma}$, the smallest closed, connected, 
complex (resp. real) subgroup $\widehat{G_1}$ (resp. $\widehat{G_0}$) of $\widehat{G_a} (\Gamma)$
containing $\widehat{\Gamma}$ is well-defined.   
By construction $\widehat{G_0}/\widehat{\Gamma}$ is compact, e.g., see  \cite[Theorem 2.1]{Rag72}.  
Set $G_1 :=\pi(\widehat{G_1})$ and $G_0 :=\pi(\widehat{G_0})$ and consider the CRS manifold $(G_1 , G_0 , \Gamma)$.       
Note that the homogeneous CR--manifold $G_0/\Gamma$ 
fibers as a $T$--bundle over $S^1\times S^1$.  
In order to understand the complex structure on the base $S^1\times S^1$   
of this fibration consider the diagram     
\begin{equation*}
  \begin{array}{rcccc}  
  \widehat{G}_0/ \widehat{\Gamma} & \subset &  \widehat{G}_1 / \widehat{\Gamma}  &   
  \subseteq  & \widehat{G}_a(\Gamma) / \widehat{\Gamma}  \\  
  & & & & \\  
  || & & ||  &  &  ||   \\ 
  & & & & \\  
   G_0/\Gamma & \subset &  G_1 / \Gamma  &   \subseteq  & G_a(\Gamma) /\Gamma \\  
   & & & & \\  
   T \downarrow  & & T \downarrow & & \downarrow T \\   
   & & & & \\  
   S^1 \times S^1 \; = \; G_0 / G_0 \cap (J^0 \cdot \Gamma) & \subset & G_1 / J^0\cdot\Gamma  & \subseteq  & G_a / J^0\cdot\Gamma
 \end{array}
\end{equation*}   
As observed in \cite[Theorem 1]{GO08}, the manifold $G_a / J^0\cdot\Gamma$   
is a holomorphically separable solvmanifold and thus is Stein \cite{HO86}.  
So $G_1 / J^0 \cdot \Gamma$ is also Stein and thus 
$G_0 / G_0 \cap (J^0 \cdot \Gamma) $ is totally real in $G_1/J^0\cdot \Gamma$.  
The corresponding complex orbit $G_1 / J^0\cdot\Gamma $ is then 
biholomorphic to $\mathbb C^*\times\mathbb C^*$.  
It now follows by \cite[Theorem 6.14]{GO08} that a finite covering of $G_1/\Gamma$ splits as a product 
of a torus with $\mathbb C^*\times\mathbb C^*$ and, in particular, that 
a subgroup of finite index in $\Gamma$ is Abelian.  
 
\medskip    
Now set $A := \{ \ \exp\ t\xi \ | \ t \in\mathbb C \ \}$, where $\xi \in \mathfrak g \setminus \mathfrak n$ 
and $\mathfrak n$ is the Lie algebra of $N^0$.  
Then $G = A \ltimes N^0$ and any $\gamma\in\Gamma$ can be written as $\gamma = \gamma_A . \gamma_N$ 
with $\gamma_A\in A$ and $\gamma_N\in N$.   
The fiber $G/\Gamma \to G/N$ is the $N^0$-orbit of the neutral point and $\Gamma $
acts on it by conjugation.  
Since $N/\Gamma$ is K\"{a}hler and has two ends, it follows by    
\cite[Proposition 1]{GOR89} that a finite covering of $N/\Gamma$ is biholomorphic to a product  
of the torus and $\mathbb C^*$.   
(By abuse of language we still call the subgroup having finite index $\Gamma$.)    
Now the bundle $G/\Gamma \to G/N$ is associated to the bundle 
\[  
        \mathbb C \; = \; G/N^0 \; \longrightarrow \; G/N \; = \; \mathbb C^*     
\]  
and thus $G/\Gamma = \mathbb C \times_{\rho} (T \times \mathbb C^*)$,     
where $\rho : N/N^0 \to {\rm Aut } (T \times \mathbb C^*)$ is the adjoint representation. 
Since $\Gamma$ is Abelian, this implies $\gamma_A$ acts trivially on $\Gamma_N := \Gamma \cap N^0$.   
Now suppose $J$ has complex dimension $k$.   
Then $\gamma_A$ is acting as a linear map on $N^0 = \mathbb C \ltimes J^0 = \mathbb C^{k+1}$ and 
commutes with the additive subgroup $\Gamma_N$ that has rank $2k+1$ and spans $N^0$ as a linear space.        
This implies $\gamma_A$ acts trivially on $N^0$ and, as a consequence, the triviality of 
a finite covering of the bundle, as required.    
\end{proof}

\subsection{The classification for discrete isotropy}\label{maintheorem}

In the following we classify K\"{a}hler $G/\Gamma$ when $\Gamma$ is discrete and $d_X \le 2$.   
Note that $d_X=0$ means $X$ is compact and this is the now classical result of 
Borel--Remmert \cite{BR62} and the case $d_X = 1$ corresponds to $X$ having 
more than one end and  was handled in \cite{GOR89}.          

\begin{thm} [\cite{Ahm13}]  \label{disc}
Let $G$ be a connected simply connected complex Lie group and $\Gamma$ a discrete 
subgroup of $G$ such that $X := G/\Gamma$ is K\"{a}hler and $d_X \le 2$.  
Then $G$ is solvable and a finite covering of $X$ is biholomorphic to a product $C \times A$, where $C$ 
is a Cousin group and $A$ is $\{ e \}, \mathbb C^*$, $\mathbb C$, or $(\mathbb C^*)^2$.   
Moreover, $d_X = d_C + d_A$.   
\end{thm}    

\begin{proof}
By Proposition \ref{solvgp} $G$ is solvable.       
If $\mathcal{O}(X)\cong\mathbb C$, then $X$ is a Cousin group \cite{OR88}.   
Otherwise,  $\mathcal{O}(X)\neq\mathbb C$ and let   
\begin{displaymath}
    \xymatrix{
          G/\Gamma \ar[r]^{J/\Gamma} &  G/J}
\end{displaymath} 
be the holomorphic reduction.         
Its base $G/J$ is Stein  \cite{HO86}, its fiber $J/\Gamma$ is biholomorphic to a Cousin group  \cite{OR88}, and    
a finite covering of the bundle is principal \cite{GO08}.       
Since $G/J$ is Stein, by Proposition \ref{SerreStein} one has 
\[   
       \dim_{\mathbb C} G/J \; \leq \; d_{G/J} \; \leq \; d_X\; \leq  \; 2  .     
\]       

\medskip
If $d_X=1$, then $d_{G/J}=1$ and $G/J$ is biholomorphic to $\mathbb C^*$.
A finite covering of this bundle is principal, with the connected Cousin group as structure group,  
and so is holomorphically trivial   \cite{Gra58}.   
If $d_X=2$, the Fibration Lemma implies $G/J\cong \mathbb C$, 
$\mathbb C^*$, $\mathbb C^*\times \mathbb C^*$ or a complex Klein bottle \cite{AG94}.   
The case of $\mathbb C^*$ is handled as above and a torus bundle over $\mathbb C$ is trivial by Grauert's Oka Principle \cite{Gra58}.   
Finally, since a Klein bottle is a 2-1 cover of $\mathbb C^*\times\mathbb C^*$,   
it suffices to consider the case $\mathbb C^*\times\mathbb C^*$.  
That case is handled by Proposition \ref{torbdle}.    \end{proof}

\section{Proof of the Main Result when  ${\mathcal O}(X) = \mathbb C$}

\begin{proof}        
Let $\pi: G/H \to G/N$ be the normalizer fibration.  
We claim that $G'$ acts transitively on $Y := G/N$.   
In order to see this consider  the commutator fibration $G/N \to G/N\cdot G'$  which exists by Chevalley's theorem, see \cite{Che51}, and 
recall that its base $G/N\cdot G'$ is an Abelian affine algebraic group that is Stein \cite{HO81} and so must be a point.     
Otherwise, one could pullback non--constant holomorphic functions to $X$ contradicting the assumption ${\mathcal O}(X) = \mathbb C$.   
Since $G'$ acts on $G/N$ as an algebraic group of transformations and $d_{G/N} \le d_X = 2$,       
there is a parabolic subgroup $P$ of $G'$ containing $N\cap G'$ (see  \cite{Akh83} or Theorem \ref{alggps}) 
and we now consider the fibration $Y = G'/N\cap G' \to G'/P$.  

\medskip  
First we assume $G/N$ is compact and thus a flag manifold, i.e., $N\cap G'=P$ and suppose ${\mathcal O}(N/H) = \mathbb C$.  
Then $N/H$ is a Cousin group by Theorem \ref{disc}.      
The structure in this case is given in  Proposition \ref{Cousinoverflag}: $X$ fibers as a $(\mathbb C^*)^k$--bundle 
over a product $C \times Q$ with $d_C + k = 2$ with $C$ a Cousin group and $Q$ a flag manifold.     

\medskip 
Next suppose $G/N$ compact and ${\mathcal O}(N/H) \not= \mathbb C$  with holomorphic reduction $N/H \to N/I$. 
The possibilities for $N/H$ have  been presented in Theorem \ref{disc};   
$N/H$ is a solvmanifold,  $N/I$ is Stein \cite{HO86} and $I/H$ is a Cousin group \cite{OR88}.  
So we get the following cases and we claim that none of these can actually occur:   
\begin{itemize}
\item []  (i) 
 $N/I = \mathbb C^*$ and $I/H =: C$ is a Cousin group of hypersurface type.  
\item [] (ii) 
$N/I = (\mathbb C^*)^2$ and $I/H = T$ is a torus.  
\item [] (iii) $N/I = \mathbb C$ and $I/H = T$ is a torus.  
\end{itemize}  
   
\medskip 
In (i) and (ii) we have ${\mathcal O}(G/I) \not=\mathbb C$ which contradicts ${\mathcal O}(X) = \mathbb C$.  
The same contradiction occurs in (iii).    
Let $I := {\rm Stab}_G(T)$.  
Then $N$ normalizes $I$, since the partition of $T\times\mathbb C$ by the cosets of $T$ is $N$--invariant.  
Thus $N/I \cong\mathbb C$ as groups.  
Now $S$ is not transitive on $G/I$ since this would give an affine bundle $S/S\cap I \to S/P$ with 
$S\cap I$ not normal in $P$ \cite[Proposition 1 in \S 5.2]{AHR85}.   
Therefore the $S$--orbits in $G/I$ are sections, 
the bundle $G/I \to G/N$ is holomorphically trivial and ${\mathcal O}(G/I) \not=\mathbb C$. 
However, this once again contradicts ${\mathcal O}(X) = \mathbb C$.

\medskip  
Now suppose $d_{G/N} = 1$.  
As noted above, $G'$ acts algebraically and transitively on $G/N$ and  
$G/N$ is an affine cone minus its vertex or simply $\mathbb C^*$.     
Thus ${\mathcal O}(G/N) \not= \mathbb C$ contradicting ${\mathcal O}(X) =\mathbb C$.   

\medskip  
Suppose $d_{G/N}=2$ and a finite covering of $P/N\cap G'$ is biholomorphic to $ (\mathbb C^*)^2$.  
An intermediate fibration has base an affine cone minus its vertex and there exist       
non--constant holomorphic functions with the same contradiction as before.     

\medskip  
Finally, suppose $d_{G/N} = 2$ and $P/N\cap G' = \mathbb C$.  
Then there are two possibilities and we first   suppose $S$ acts transitively on $Y$.     
Let $x_0$ be any point in $X$ and $z_0\in S/I$ be its projection and consider the $S$--orbit 
$S . x_0 = S/S \cap H$ through the point $x_0$.  
There is an induced fibration  
\[  
    \begin{array}{ccc}  S/S\cap H & \hookrightarrow & G/H \\  
                                   & & \\  
                                    A \downarrow & & \downarrow F_{z_0}  \\  
                                    & & \\  
                                    S/I & = & G/N  \end{array}     
\]  
with $I$ an algebraic subgroup of $S$ and $F_{z_0} = T$ a compact complex torus by the Fibration Lemma.   
Since $S/S\cap H$ is K\"{a}hler, $S\cap H$ is an algebraic subgroup of $S$.  
Then the holomorphic bundle $S/S\cap H \to S/I$ has the algebraic variety $I/S\cap H$ as fiber $A$ 
and this is a closed subgroup of the torus $T$.  
But this is only possible if $I/S\cap H$ is finite. 
Since we have the fibration $S/I \to S/P$ with $P/I = \mathbb C$ and $S/P$ a flag manifold, we see 
that $S/I$ is simply connected.   
As a consequence, every $S$--orbit in $X$ is a holomorphic section and 
we conclude that $X$ is the product of $T$ and $S/I$.   

\medskip    
Otherwise, $S$ does not act transitively on $G/N$.  
The radical $R_{G'}$ of $G'$ is a unipotent group acting algebraically on $G/N$ yielding a fibration 
\[  
              Y \; = \; G/N \; \stackrel{F}{\longrightarrow} \; G/N\cdot R_{G'} 
\]   
where $F = \mathbb C^p$ with $p > 0$.     
The Fibration Lemma and the assumption $d_X = 2$ imply $N/H$ is compact, thus a torus, 
$F = \mathbb C$ and $Z := G/N\cdot R_{G'}$ is compact and thus a flag manifold.         
Now $Y \not= \mathbb C \times Z$ because one would then have   
${\mathcal O}(Y) \not=\mathbb C$ contrary to ${\mathcal O}(X) =\mathbb C$.       
So $Y$ is a non--trivial line bundle over $Z$ and there are two $S$--orbits in $Y$,  
a compact one $Y_1$ which is the zero section of the line bundle and is biholomorphic to $Z$ and an open one $Y_2$.    
The latter holds, since the existence of another closed orbit would imply the triviality of the 
$\mathbb C^*$--bundle $Y \setminus Y_1$ over $Z$.  
We write $X$ as a disjoint union $X_1 \cup X_2$ with $X_i := \pi^{-1}(Y_i)$ for $i=1,2$.  
Then $X_1$ is a K\"{a}hler torus bundle over $Z$ and is trivial by \cite{BR62}.  
And a finite covering of $X_2$ is also trivial since $X_2$ is K\"{a}hler and satisfies $d_{X_2} = 1$ \cite[Main Theorem, Case (b)]{GOR89}.    
Note that  the $S$--orbits in $X_1$ (resp. $X_2$) are holomorphic sections of the torus bundle lying over the corresponding   
$S$--orbit  $Y_1$ (resp. $Y_2$).         

\medskip  
Let $x_2 \in X_2$ and  $M_2 := S . x_2$.   
Since $X$ is K\"{a}hler, the boundary of $M_2$ consists of $S$--orbits of strictly smaller dimension 
\cite[Theorem 3.6]{GMO11}, and for dimension reasons, these necessarily lie in $X_1$.  
Let $M_1 \subset \overline{M_2}$ be such an $S$--orbit in $X_1$ and let $p \in M_1$.  
As observed in the previous paragraph, $M_1$ is biholomorphic to $Y_1$ which is a flag manifold.  
Therefore, $M_1 = K . p = K/L$, where $K$ is a maximal compact subgroup of $S$ and $L$ is the   
corresponding isotropy subgroup at the point $p$ and is compact.  
The $L$--action at the $L$--fixed point $p$ can be linearized, i.e.    
this $L$--action leaves invariant the complex vector subspaces $T_p(K/L)$ and $T_p(\pi^{-1}(\pi(p)))$ as well as 
a complementary complex vector subspace $W$ of $T_p(X)$.   
Now $W \oplus T_p(K/L)$ is the tangent space to $\overline{M_2}$ and is a complex vector space.   
This implies that $M_2$ is the unique $S$--orbit that contains $M_1$ in its closure and also that 
$\overline{M_2} = M_2 \cup M_1$ is a complex submanifold of $X$ that is a holomorphic section of 
the bundle $\pi:X \to Y$.  
This bundle is thus trivial and $X$ is biholomorphic to the product $T \times Y$.     

 \medskip   
This completes the classification if ${\mathcal O}(X) = \mathbb C$.  \end{proof}

\section{Proof of the Main Result when ${\mathcal O}(X) \not= \mathbb C$}    
   
We first prove a generalization of Proposition \ref{torbdle} for arbitrary isotropy.  

\begin{prop} \label{solvcase}   
Let $G$ be a connected, simply connected, solvable complex Lie group and $H$ a closed complex subgroup of $G$   
with $G/H$ K\"{a}hler and $G/H \to G/J $  its holomorphic reduction with fiber $T = J/H$ a compact complex torus  
and base $G/J = (\mathbb C^*)^2$.  
Then a finite covering of $G/H$ is biholomorphic to $T \times (\mathbb C^*)^2$.  
\end{prop}  

\begin{proof}  
If $H^0$ is normal in $G$, then this is Proposition \ref{torbdle}.  
Otherwise, let $N := N_G(H^0)$ and consider $G/H \to G/N$.  
Because $G/N$ is an orbit in some projective space,  
$G/N$ is holomorphically separable and the map $G/H \to G/N$ factors through the holomorphic reduction, i.e., $J \subset N$.   
We first assume that $J=N$ and consider now $\widehat{N} := N_G(J^0)$.  
Then the argument given in the third last paragraph of the proof of Proposition \ref{torbdle} shows that 
$G/\widehat{N} = \mathbb C^*$ and $\widehat{N}/J = \mathbb C^*$.   
But then (a finite covering of)  $\widehat{N}/H$ is isomorphic to $\mathbb C^* \times T$, see \cite[Proposition1]{GOR89}, i.e., 
that $H^0$ is normal in $\widehat{N}$.     
This gives the contradiction that $\widehat{N} = N$ while $\dim \widehat{N} > \dim J = \dim N$.  

\medskip 
So we are reduced to the setting where, after going to a finite covering if necessary,  $G/N = \mathbb C^*$ and   
$N/H \cong \mathbb C^* \times T$ is an Abelian complex Lie group,   
since $N/H$ is K\"{a}hler with two ends, see \cite[Proposition 1]{GOR89}.           
We have the diagram 
\[  
    \begin{array}{rcl}  X \; = \; G/H & \stackrel{T}{\longrightarrow} & G/J \; = \; \mathbb C^* \times \mathbb C^* \\  
    & & \\ 
    \searrow & & \swarrow \\  
    & & \\ 
    & G/N_{G}(H^0) \; = \; \mathbb C^* & \end{array} 
\]       
Since the top line is the holomorphic reduction and $X$ is K\"{a}hler,   
a finite covering of this bundle is a $T$--principal bundle, see \cite[Theorem 1]{GO08}.   
Choose $\xi \in \mathfrak g \setminus \mathfrak n$ and set $A := \exp \langle \xi \rangle_{\mathbb C}$    
and $B := N_G(H^0)/H \cong \mathbb C^* \times T$.  
Then the group $\widehat{G} := A \ltimes B$ acts holomorphically and transitively on $X$, where 
$A$ acts from the left and $B$ acts from the right on the principal $\mathbb C^*\times T$--bundle $G/H \to G/N$.  
For dimension reasons the isotropy of this action is discrete and the result now follows by Proposition \ref{torbdle}.    
\end{proof}   

%%%%%%%%%%%%%%%%%  

\begin{proof}[Proof of the Main Result when ${\mathcal O}(X) \not= \mathbb C$]        
Let $G/H \to G/J$ be the holomorphic reduction.  
In \cite{AG94} there is a list of the possibilities for $G/J$ which was also given above     
in the proof of Proposition \ref{solvgp}.  
We now employ that list to determine the structure of $X$.  

\medskip  
Suppose $G/J = \mathbb C$.  
By the Fibration Lemma $J/H$ is compact, K\"{a}hler, and so biholomorphic to the product 
of a torus $T$ and a flag manifold $Q$.   
Thus $X = T \times Q \times \mathbb C$ by \cite{Gra58}.     

\medskip    
Suppose $G/J$ is an affine quadric.  
By the Fibration Lemma we again have  $J/H = T \times Q$.     
Then $X$ is biholomorphic to a product, since $G/J$ is Stein and is simply connected \cite{Gra58}.  

\medskip    
If $G/J$ is the complement of the quadric curve in $\mathbb P_2$,     
then a two--to--one covering of $G/J$ is the affine quadric and the pullback of $X$ to that space 
is again a product, as in the previous case.   

\medskip    
Suppose $G/J$ is a  $\mathbb C^*$--bundle over an affine cone minus its vertex and so $d_{G/J}=2$   
and let $G = S\ltimes R$ be a Levi--Malcev decomposition of $G$.         
By the Fibration Lemma $J/H$ is compact and $J/H$ inherits a K\"{a}hler structure from $X$.        
If we set $N := N_J(H^0)$, then the normalizer fibration $J/H \to J/N$ is a product with 
$N/H = T$ and $J/N = Q$, i.e., $J/H = T \times Q$, see \cite{BR62}.  
Now $S$ is acting transitively on $G/N$ and by pulling the bundle back to the universal covering of $G/J$, 
without loss of generality, we may assume that $G/N$ is simply connected.  
Consider the bundle $G/H \to G/N$.  
Its fibers are biholomorphic to $T$ and the intersection of every $S$--orbits with this fiber $N/H$ is biholomorphic to $Q$.       
Since the $S$--orbits map onto $G/J$, it    
follows that the $S$--orbits are sections of the bundle $G/H \to G/N$ and the bundle is holomorphically trivial.    
Hence a finite covering of $X$ is biholomorphic to $T \times (S/S\cap N)$.     
Example \ref{exdoubleCstar/flag} shows that the $Q$--bundle $S/S\cap N \to S/S\cap J$ is not necessarily trivial.  

\medskip     
Suppose the base of the holomorphic reduction of $G/H$ is $G/J = (\mathbb C^*)^2$, i.e., there 
is no flag manifold involved.  
We assume that $G$ is simply connected and has a Levi--Malcev decomposition $G = S \ltimes R$.  
Let $p: G/H \to G/J$ be the bundle projection map.  
We now fiber $J/H$ in terms of its $S$--orbits first.   
Then the partition of $X$ by the fibers of the map $p$ is $S$--invariant.  
And $J/H$ is compact and K\"{a}hler and thus biholomorphic to a product $T \times Q$, where the flag 
manifold passing through each point is an $S$--orbit.    
As a consequence, the $S$--orbit in $X$ are all closed and  biholomorphic to $Q$ and we have 
a homogeneous fibration $G/H \to G/S\cdot H$.    
Since all holomorphic functions are constant on the fibers of this fibration, the holomorphic reduction 
factors through this bundle projection and the base $G/S\cdot H = R/R\cap H$ then fibers as $T$--bundle over 
the base of the holomorphic reduction $G/J$ and is K\"{a}hler \cite{Bla56}.   
As a consequence of Proposition \ref{solvcase}, $X = Q \times T \times (\mathbb C^*)^2$.    

\medskip   
Finally, suppose $G/J$ is an affine cone minus its vertex (or possibly $\mathbb C^*$) and thus $d_{G/J}=1$.  
Now by the Fibration Lemma the fiber $J/H$ satisfies $d_{J/H}=1$ and is K\"{a}hler.  
The classification given in \cite[Main Theorem]{GOR89} now applies.     
First assume ${\mathcal O}(J/H) = \mathbb C$.  
Then by \cite[Proposition 5]{GOR89} the normalizer fibration $J/H \to J/N$ realizes 
$J/H$ as a Cousin bundle over a flag manifold.  
By Proposition \ref{Cousinoverflag} there is a closed complex subgroup $I$ of $J$    
containing $H$ with $I/H = \mathbb C^*$ and $J/I =T \times Q$, a product of a torus and a flag manifold.   
The $S$--orbits in $G/I$ are sections and the torus bundle splits as a product $Y := T \times S/S \cap I$.  
Thus $X$ fibers as a $\mathbb C^*$--bundle over $T \times Y$.  
Example \ref{Qovercone} shows that $Y$ itself need not be a product.     
Next assume ${\mathcal O}(J/H) \not= \mathbb C$.  
Then by \cite[Proposition 3]{GOR89}      
there exists a closed complex subgroup $I$ of $J$ containing $H$ such that 
a finite covering of $J/H$ is biholomorphic to $I/H \times J/I$, 
where $I/H = T$ is a torus and $Z := J/I$ is an affine cone minus its vertex.   
Then the fibration $G/H \to G/I$ has $T$ as fiber and $S$ is transitive on its base and  
since $T$ is compact, there is a K\"{a}hler structure on $G/I$, see \cite{Bla56}.   
Now by the Fibration Lemma we have $d_{G/I} = 2$ and thus there exists a parabolic 
subgroup $P$ of $G$ containing $I$ such that $P/I$ is isomorphic to $(\mathbb C^*)^2$,  
see \cite{Akh83} or Theorem \ref{alggps}.  
But then $G/I$ is holomorphically separable ($G/I$ can be realized as a $\mathbb C^*$--bundle over an affine cone minus its vertex) 
and so $G/I$ is the base of the holomorphic reduction of $G/H$.    
But this contradicts the assumption that $G/J$ is the base of the holomorphic reduction of $G/H$.  
Note that it is not possible that $J=I$, since ${\mathcal O}(J/H) \not= \mathbb C$ implies $\dim J/I > 0$.  
This case does not occur.   

\medskip  
This completes the classification when ${\mathcal O}(X) \not= \mathbb C$.   
\end{proof}    

\section{Examples} 

We now give non--trivial examples that can occur in the classification.

\begin{example} \label{exCousin/flag}  
The manifolds that occur   
in Proposition \ref{Cousinoverflag} need not be biholomorphic to a product of an $S$--orbit times an orbit of the center.  
For $k=d_X=1$,  let $\chi : B \to \mathbb C^*$ be a non--trivial character, 
where $B$ is a Borel subgroup of $S := SL(2,\mathbb C)$.  
Let $C$ be a non--compact 2--dimensional Cousin group.  
Then $C$ fibers as a $\mathbb C^*$--bundle over an elliptic curve $T$ and let $B$ act on $C$ via the character $\chi$.  
Set $X := S \times_{B} C$.  
Then $X$ fibers as a principal $C$--bundle over $S/B$ and is K\"{a}hler, but neither this bundle 
nor the corresponding $\mathbb C^*$--bundle is trivial.   
\end{example}

\begin{example}   \label{ex1}  
Let $S := SL(3,\mathbb C)$ and    
\[        
     H \; := \; \left\{ \left( 
        \begin{matrix}* & 0 & * \\  0 & * & *  \\ 0 & 0 & * \end{matrix}  
      \right) \right\}  \quad \subset \quad B \quad \subset \quad   
   P \;  :=  \;  \left\{ \left(  \begin{matrix}*&*&*\\0&*&*\\0&*&*\end{matrix}      \right) \right\} ,   
\]  
where $B$ is the Borel subgroup of $S$ consisting of upper triangular matrices.    
Then $S/H \to S/B$ is an affine $\mathbb C$--bundle over the flag manifold $S/B$.  
Now consider the fibration $S/H \to S/P$.  
Its fiber is $P/H = \mathbb P_2 \setminus \{ \mbox{point} \}$ and all holomorphic functions on $S/H$ 
are constant along the fibers by Hartogs' Principle and so must come from the base $S/P = \mathbb P_2$.      
But the latter is compact and so ${\mathcal O}(S/P)=\mathbb C$ and, as a consequence, we see that  ${\mathcal O}(S/H) = \mathbb C$.   
Thus $S/H$ is an example of a space that can be the base of the normalizer fibration in the second paragraph of the proof of the Main 
Theorem in  \S 4.  
\end{example}     

\begin{example}  \label{ex2}  
The space $Y = \mathbb P_n \setminus \{ z_0 \}$ is an example for 
the base of the normalizer fibration in  the third paragraph of the proof of the Main Theorem in \S 4.    
For $n=2$ we have $Y = P/H$ with the groups given in Example \ref{ex1}.       
\end{example}  

\begin{example} \label{Qovercone}  
Consider the subgroups of $S := SL(5,\mathbb C)$ defined by 
\[ 
     H \; := \; \left\{ \left( 
        \begin{matrix}1 & * & * & * & *  \\  0 & * & * & * & *  \\ 0 & 0 & * & * & * \\ 0 & 0 & 0 & * & * \\ 0 & 0 & 0 & * & * \end{matrix}  \right) \right\}   
\quad \subset  \quad 
              P  \; := \; \left\{ \left( 
        \begin{matrix}* & * & * & * & *  \\  0 & * & * & * & *  \\ 0 & * & * & * & * \\ 0 & 0 & 0 & * & * \\ 0 & 0 & 0 & * & * \end{matrix}     \right) \right\}     
\]  
and   
\[            
   J  \; := \; P' \; = \; \left\{ \left( 
        \begin{matrix}1 & * & * & * & *  \\  0 & * & * & * & *  \\ 0 & * & * & * & * \\ 0 & 0 & 0 & * & * \\ 0 & 0 & 0 & * & * \end{matrix}   \right) \right\}        
\]     
Then $J/H = \mathbb P_1$ and $P/J = \mathbb C^*$ and $S/P = Q$ is a flag manifold that can be fibered as a ${\rm Gr}(2,4)$--bundle 
over $\mathbb P_4$.     
We have the fibrations   
\[ 
              S/H \; \stackrel{\mathbb P_1}{\longrightarrow} \; S/J  \; \stackrel{\mathbb C^*}{\longrightarrow} \; S/P \; = \; Q .   
\]     
Note that $S/J$ is holomorphically separable due to the fact that it can be equivariantly embedded as an affine cone minus its vertex in some projective space       
and since $J/H$ is compact, $S/J$ is the base of the holomorphic reduction of $S/H$.  
Because the fibration of $S/H$ is not trivial, this shows that the $S$--orbit that can occur in the proof in \S 5 need not split as a product.  
\end{example}   
   
\begin{example}  \label{exdoubleCstar/flag}   
One can take a parabolic subgroup similar to the one in the previous example so that its center has dimension two and 
create an example which fibers as a non--trivial flag manifold over a $(\mathbb C^*)^2$--bundle over a flag manifold.   
Since this is similar to the construction in Example \ref{Qovercone}, we leave the details to the reader.    
\end{example}


\begin{thebibliography}{99}

\bibitem{AK01}  
  Y. Abe and K. Kopfermann,  
  Toroidal groups. Line bundles, cohomology and quasi-abelian varieties. 
  Lecture Notes in Mathematics, 1759. Springer-Verlag, Berlin, 2001.    

\bibitem{Abe76}
  H. Abels, \;\emph{Proper transformation groups.} In: ``Transformation groups. 
  Proc. Conf., Univ. Newcastle upon Tyne, Newcastle upon Tyne, 1976'', pp. 237 -- 248. 
  London Math. Soc. Lecture Note Series, No. 26, Cambridge Univ. Press, Cambridge, 1977.
  
\bibitem{Abe82}  
  H. Abels,\;\emph{Some topological aspects of proper group actions; noncompact dimension of groups.}     
  J. London Math. Soc. (2) \textbf{25} (1982), no. 3, 525 -- 538.

\bibitem{Ahm13}  
   S. Ahmadi, A classification of homogeneous K\"{a}hler manifolds with discrete isotropy and 
   top non--vanishing homology in codimension two.  Ph.D. Dissertation, University of Regina, 2013.  

\bibitem{Akh77}
  D. N. Ahiezer, \;\emph{Dense orbits with two endpoints.}
  (Russian) Izv. Akad. Nauk SSSR Ser. Mat. \textbf{41} (1977), no. 2, 308 -- 324; 
  translation in  Math. USSR-Izv. \textbf{11} (1977), no. 2, 293 -- 307 (1978).

\bibitem{Akh83}
  D. N. Akhiezer, \; \emph{Complex $n$--dimensional homogeneous spaces homotopically equivalent to
  $(2n-2)$--dimensional compact manifolds}.
  Selecta Math. Sov. \textbf{3}, no.3 (1983/84), 286 -- 290.   
  
\bibitem{AG94}
 D. N. Akhiezer and B. Gilligan,\;\emph{On complex homogeneous spaces with top homology in 
 codimension two},  Canad. J. Math.  \textbf{46} (1994), 897 -- 919.   
 
\bibitem{Akh95} 
  D. N. Akhiezer, \; ``Lie Group Actions in Complex Analysis'', Aspects of Mathematics,     
  Friedr. Vieweg \& Sohn Verlagsgesellschaft mbH, Braunschweig/Wiesbaden, 1995.   

\bibitem{AT70}  
   L. Auslander and R. Tolimieri, \; \emph{Splitting theorems and the structure of solvmanifolds.} 
   Ann. of Math. (2) \textbf{92} (1970), 164 -- 173.  

\bibitem{AHR85}  
    H. Azad, A. Huckleberry, and W. Richthofer,     
    \emph{Homogeneous CR-manifolds.}  J. Reine Angew. Math. \textbf{358} (1985), 125 -- 154.

\bibitem{BO88}
  R. Berteloot and K. Oeljeklaus, \;\emph{Invariant plurisubharmonic functions and hypersurfaces 
  on semi-simple complex Lie groups.}  Math. Ann. \textbf{281} (1988), no. 3, 513 -- 530.    

\bibitem{Bla56}
 A. Blanchard, \;\emph{Sur les vari\'{e}t\'{e}s analytiques complexes.}   
 Ann. Sci. Ecole Norm. Sup. (3) \textbf{73} (1956), 157 -- 202. 

\bibitem{BR62} 
   A. Borel and R. Remmert, \; \emph{\"{U}ber kompakte homogene K\"{a}hlersche Mannigfaltigkeiten.}   
   (German) Math. Ann. \textbf{145} (1961/1962),  429 -- 439.  

\bibitem{Che51}
  C. Chevalley, \; ``Th\'{e}orie des Groupes de Lie II. Groupes Alg\'{e}briques'',  
  Hermann, Paris, 1951.      
  
\bibitem{DN88}   
 J. Dorfmeister and K. Nakajima, \; \emph{The fundamental conjecture for homogeneous K\"{a}hler manifolds.}  
 Acta Math. \textbf{161} (1988), no. 1-2, 23 -- 70.

\bibitem{FHW06}  
   G. Fels, A. Huckleberry, and J. Wolf, 
   ``Cycle spaces of flag domains. A complex geometric viewpoint''.
    Progress in Mathematics, \textbf{245}. BirkhŠuser Boston, Inc., Boston, MA, 2006.     

\bibitem{GOR89}
  B. Gilligan, K. Oeljeklaus, and W. Richthofer,\;\emph{Homogeneous complex manifolds with more than one end.}
  Canad. J. Math. \textbf{41} (1989), no. 1, 163 -- 177.   

\bibitem{Gil04}
  B. Gilligan,  \;\emph{Invariant analytic hypersurfaces in complex Lie groups.},
  Bull. Austral. Math. Soc. \textbf{70} (2004), no. 2, 343 -- 349.     
 
\bibitem{GO08}
   B. Gilligan and  K. Oeljeklaus, \;\emph{Two remarks on K\"{a}hler homogeneous manifolds},
  Ann Fac. Sci. Toulouse Math. (6) \textbf{17} (2008), 73 -- 80.  

% \bibitem{GO11} 
%   B. Gilligan and K. Oeljeklaus,   \;\emph{Compact CR-solvmanifolds as K\"{a}hler obstructions.}   
%  Math. Z. 269 (2011), 179 -- 191.    

\bibitem{GMO11}
  B. Gilligan, C. Miebach, and K. Oeljeklaus, \;\emph{Homogeneous K\"{a}hler and Hamiltonian manifolds}, 
  Math. Ann. \textbf{349} (2011), 889 -- 901. 

\bibitem{Gra58} 
  H. Grauert,  \;\emph{Analytische Faserungen \"{u}ber holomorph-vollst\"{a}ndigen R\"{a}umen.}   
  Math. Ann. \textbf{135} (1958),  263 -- 273.   
  
 \bibitem{HM64}
  G. Hochschild and G. D. Mostow,\;\emph{ On the algebra of representative functions 
  of an analytic group. II,} 
  Amer. J. Math., \textbf{86} (1964), 869 -- 887.     
   
\bibitem{HO81}
   A. T. Huckleberry and E. Oeljeklaus, \;\emph{Homogeneous Spaces from a Complex Analytic Viewpoint},  In:   
  Manifolds and Lie groups. (Papers in honor of Y. Matsushima),  
  Progress in Math., \textbf{14}, Birkhauser, Boston (1981), 159 -- 186.   

\bibitem{HO86}
  A. T. Huckleberry and E. Oeljeklaus,  \;\emph{On holomorphically separable complex solv-manifolds.} 
  Ann. Inst. Fourier (Grenoble) \textbf{36} (1986), no. 3, 57 -- 65.    

\bibitem{Hum75} 
  J. E. Humphreys, \; ``Linear algebraic groups''. Graduate Texts in Mathematics, No. 21. Springer-Verlag, New York-Heidelberg, 1975. 

\bibitem{Kod54}  
  K. Kodaira, \;\emph{On K\"{a}hler varieties of restricted type (an intrinsic characterization of algebraic varieties). }
  Ann. of Math. (2) \textbf{60} (1954), 28 -- 48.      
  
\bibitem{Loe85}  
  J.-J. Loeb, \ \emph{Action d'une forme r\'{e}elle d'un groupe de Lie complexe sur les fonctions plurisousharmoniques}.   
  Ann. Inst. Fourier (Grenoble) \textbf{35} (1985), 59 -- 97.         

\bibitem{Mat57}  
   Y. Matsushima, \; \emph{Sur les espaces homog\`{e}nes k\"{a}hl\'{e}riens d'un groupe de Lie r\'{e}ductif.} (French) 
   Nagoya Math. J. \textbf{11} (1957), 53 -- 60.
    
\bibitem{Mos54}  
   G. D. Mostow, \; \emph{Factor spaces of solvable groups.}  Ann. of Math. (2) \textbf{60} (1954). 1 -- 27.     
    
\bibitem{Mos55} 
  G. D. Mostow,   \emph{On covariant fiberings of Klein spaces. I.}  Amer. J. Math. \textbf{77} (1955) 247 -- 278.   
 \emph{Covariant fiberings of Klein spaces. II.} Amer. J. Math. \textbf{84} 1962 466 -- 474.    
    
\bibitem{OR87}
  K. Oeljeklaus and W. Richthofer, \;\emph{Recent results on homogeneous complex manifolds.} In:   
 Complex analysis, III (College Park, Md., 1985–-86), 78 -- 119, 
 Lecture Notes in Math., \textbf{1277}, Springer, Berlin, 1987.   

\bibitem{OR88}
  K. Oeljeklaus and W. Richthofer,\;\emph{On the structure of complex solvmanifolds.} 
  J. Differential Geom. \textbf{27} (1988), no. 3, 399 -- 421.   

\bibitem{Rag72} 
   M. S. Raghunathan,  ``Discrete subgroups of Lie groups'',    
   Ergebnisse der Mathematik und ihrer Grenzgebiete, Band 68. Springer--Verlag, New York--Heidelberg, 1972.

\bibitem{Ser53}
  J.-P. Serre, \;\emph{Quelques probl\`{e}mes globaux relatifs aux vari\'{e}t\'{e}s de Stein}, In: Coll. sur
  les fonctions de plusieurs variables, Bruxelles (1953), 57 -- 68.  

\end{thebibliography}
\end{document}